\documentclass{amsart}
\usepackage{amssymb}
\usepackage[abbrev]{amsrefs}

\newcommand{\R}{\mathbb R}

\newcommand{\co}{\colon}
\newcommand{\pd}{\partial}

\renewcommand{\phi}{\varphi}

\newtheorem{lemma}{Lemma}[section]
\newtheorem{theorem}[lemma]{Theorem}
\newtheorem{proposition}[lemma]{Proposition}

\theoremstyle{remark}

\begin{document}

\title{On Helly's theorem in geodesic spaces}
\author{Sergei Ivanov}
\email{svivanov@pdmi.ras.ru}
\address{Sergei Ivanov: St.~Petersburg Department of Steklov Mathematical Institute,
Russian Academy of Sciences,
Fontanka 27, St.~Petersburg 191023, Russia}
\thanks{Supported by Russian Foundation for Basic Research
grant 14-01-00062.}

\begin{abstract}
In this note we show that Helly's Intersection Theorem holds
for convex sets in uniquely geodesic spaces (in particular in CAT(0) spaces)
without the assumption that the convex sets are open or closed.
\end{abstract}



\maketitle

\section{Introduction}

The classic Helly's Intersecton Theorem asserts the following:
If $\{A_i\}$ is a finite collection of convex sets in $\R^n$
such that every subcollection consisting of at most $n+1$ 
sets has a nonempty intersection, then $\bigcap A_i\ne\emptyset$.
This theorem has a topological generalization (found by Helly himself \cite{Helly30})
where convexity is replaced by the assumption that the sets $A_i$ and
their nonempty intersections are open homology cells.
See \cite{Deb70} for a modern proof and further references.

The proof of the topological Helly's theorem extends to CAT(0) 
spaces of geometric dimension~$n$, see e.g.\ \cite[Proposition 5.3]{Kle99}
and \cite[\S3]{Farb09}. 
Thus Helly's theorem holds for open convex sets in such spaces.
Once the theorem is established for open sets, the variant with closed convex sets follows.
In $\R^n$, one can deduce the theorem for arbitrary convex sets
by picking one point in every nonempty intersection and replacing every set
by the convex hull of the marked points it contains.
However this argument does not work in CAT(0) spaces since convex hulls of finite sets are
not necessarily closed.

In this note we show that Helly's theorem holds for arbitrary
(not necessarily open or closed)
convex sets in CAT(0) and some other spaces.
Namely we prove the following.

\begin{theorem}\label{t:main}
Let $X$ be a uniquely geodesic space of compact topological dimension $n<\infty$.
Let $\{A_i\}$ be a finite collection of convex sets in $X$
such that every subcollection of cardinality at most $n+1$ has a nonempty intersection.
Then $\bigcap A_i\ne\emptyset$.
\end{theorem}

\subsection*{Definitions}
Here are the definitions of terms used in Theorem~\ref{t:main}.

A \textit{geodesic space} is a metric space $X$ such that every
two points in $X$ belong to a segment, where a \textit{segment} is
a subset isometric to a compact interval of the real line.
We say that $X$ is \textit{uniquely geodesic} if
for every $x,y\in X$ there is a unique segment $[xy]\subset X$
with endpoints at $x$ and~$y$, and $[xy]$ depends continuously
on $x$ and~$y$. Note that the continuous dependence is automatic if $X$
is proper (i.e., if all closed balls are compact).

Examples of uniquely geodesic spaces are simply connected Riemannian
and Finsler manifolds without conjugate points,
CAT(0) spaces, balls of radius $\pi/2\sqrt{\kappa}$ in CAT($\kappa$) spaces
(see \cite[\S II.1]{BH99} for definitions),
Busemann convex spaces \cite{Bus48},
simply connected
polyhedral Finsler spaces with locally unique geodesics~\cite{BI13}.

The \textit{compact topological dimension} $\dim_c X$ of $X$
is defined by
$$
 \dim_c X = \sup \{ \dim K : \text{$K\subset X$ is compact} \}
$$
where $\dim$ is the Lebesgue covering dimension.
For (locally) CAT($\kappa$) spaces, the compact topological dimension
equals the geometric dimension and a number of other
dimension-like quantities \cite{Kle99}.

A set $A\subset X$ is \textit{convex} if it contains
all segments with endpoints in $A$.

\medskip
The proof of Theorem \ref{t:main} is topological,
the only feature of convex sets used in the proof is that they 
are contractible.
See Proposition~\ref{p:topohelly} for a purely topological formulation.
Proposition~\ref{p:topohelly}
is in some ways similar to the Topological Helly Theorem \cites{Deb70},
see also more general results in \cites{Bog,Mon}.
However, known proofs of the Topological Helly Theorem
involve computation of homology groups with techniques such as
Mayer--Vietoris sequences.
This approach requires the sets in question to be open or otherwise ``nice''.
It fails to work for arbitrary convex sets such as, for example,
a Euclidean ball with a wild subset of the boundary removed.

In contrast to this, the proof of Proposition~\ref{p:topohelly}
is a combinatorial argument which does not use algebraic topology
and does not require openness.
Note that for open sets results stronger than
Proposition~\ref{p:topohelly} are known, see \cite{Mon}.

\subsubsection*{Acknowledgement}
My interest in this problem was provoked by
a MathOverflow discussion \cite{MO150351} initiated by Misha Kapovich.
I am grateful to Roman Karasev who provided some references
and encouraged me to write down this note, and to Luis Montejano
for information about recent development in the area.

\section{Proof of the theorem}

Fix $n\ge 1$ and denote by $\Delta$ the standard $(n+1)$-dimensional simplex.
By definition, $\Delta$ is the convex
hull of the standard basis $\{e_i\}_{i=1}^{n+2}$ of $\R^{n+2}$.
Let $F_i$ denote the $i$th $n$-dimensional face of~$\Delta$
(i.e., the one not containing $e_i$).
For a positive integer $m$, we denote by $[m]$ the set $\{1,2,\dots,m\}$.

The following lemma is the only place in the proof where
the dimension of the ambient space is used.

\begin{lemma}\label{l:simplex}
Let $X$ be a Hausdorff space with $\dim_c X\le n$
and $f\co\Delta\to X$ a continuous map.
Then $\bigcap_{i=1}^{n+2} f(F_i)\ne\emptyset$.
\end{lemma}

This lemma is apparently folklore.
It can be seen as a special case ($r=2$) of \cite[Theorem~1.1]{Kar12}.
Here we give a short proof based on Sperner's lemma.

\begin{proof}[Proof of Lemma \ref{l:simplex}]
We need the following fact: if $\{G_i\}_{i=1}^{n+2}$ is an open covering
of~$\Delta$ such that $G_i\cap F_i=\emptyset$ for each~$i$,
then $\bigcap G_i\ne\emptyset$. This fact is a topological variant of Sperner's lemma
and follows easily from the discrete counterpart.
Alternatively, it follows from the Knaster--Kuratowski--Mazurkiewicz lemma \cite{KKM}
which is a slightly more general statement about open or closed coverings of the simplex.

Now proceed with the proof of Lemma \ref{l:simplex}.
We may assume that $X$ is compact, otherwise take $f(\Delta)$ for~$X$.
Then $\dim X=\dim_c X\le n$.
Suppose, towards a contradiction, that $\bigcap f(F_i)=\emptyset$.
Then the sets $U_i=X\setminus f(F_i)$ form an open covering of $X$.
By the definition of the covering dimension, there exists an
open covering $\{V_j\}_{j\in J}$ refining $\{U_i\}$ and having
covering multiplicity at most $n+1$.
Let $U_i'$ be the union of those sets $V_j$ that are contained in $U_i$
but not in $U_1,\dots,U_{i-1}$. 
Since the covering multiplicity of $\{V_j\}$ is less than $n+2$,
we have $\bigcap_{i=1}^{n+2} U_i'=\emptyset$.

On the other hand, since $U_i'\subset U_i$ and $U_i\cap f(F_i)=\emptyset$,
the sets $G_i:=f^{-1}(U_i')$ satisfy
the assumptions of the topological Sperner's lemma stated above.
Hence 
$\bigcap G_i\ne\emptyset$
and therefore $\bigcap U_i'\ne\emptyset$, a contradiction.
\end{proof}

\begin{proposition}\label{p:topohelly}
Let $X$ be a contractible Hausdorff space with $\dim_c X=n<\infty$.
Let $\{A_i\}_{i=1}^m$ be a finite collection of contractible sets in $X$
such that the intersection of every subcollection is either contractible or empty.
Suppose that $m\ge n+2$ and for every set $I\subset[m]$ with $|I|=n+1$ one has
$\bigcap_{i\in I} A_i\ne\emptyset$.
Then $\bigcap_{i=1}^m A_i\ne\emptyset$.
\end{proposition}

\begin{proof}
First consider the case $m=n+2$.
For a nonempty set $I\subset [m]=[n+2]$, denote by $\Delta_I$ the convex hull
of $\{e_i\}_{i\in I}$ and let $P_I=\bigcap_{i\in[m]\setminus I} A_i$
if $I\ne[m]$. In addition, define $P_{[m]}=X$.
By the assumptions of the proposition, $P_I$ is contractible
for every nonempty set $I\subset[m]$.

We construct a continuous map
$f\co\Delta\to X$ such that $f(\Delta_I)\subset P_I$
for every $I\subset[m]$.
First for each $i\in[m]$ pick a point $f(e_i)=f(\Delta_{\{i\}})$
from the set $P_{\{i\}}$,
which is nonempty by the assumptions of the proposition.
Then extend the map by induction as follows.
Assuming that $f$ is already defined on the $(k-1)$-skeleton of~$\Delta$,
where $1\le k\le n+1$,
consider a $k$-simplex $\Delta_I$
where $I\subset[m]$, $|I|=k+1$.
Observe that $f(\pd\Delta_I)\subset P_I$
because $\pd\Delta_I=\bigcup_{i\in I}\Delta_{I\setminus\{i\}}$ and
$f(\Delta_{I\setminus\{i\}})\subset P_{I\setminus\{i\}}=P_I\cap A_i$
for every $i\in I$.
Since $P_I$ is contractible,
$f|_{\pd\Delta_I}$ can be extended to a continuous map from $\Delta_I$ to $P_I$.
Applying this extension procedure to all $k$-dimensional faces for $k=1,2,\dots,n+1$,
one gets the desired map $f\co\Delta\to X$.

By Lemma \ref{l:simplex}, we have $\bigcap_{i=1}^m f(F_i)\ne\emptyset$
where $F_i= \Delta_{[m]\setminus\{i\}}$.
By construction,
$f(F_i)\subset P_{[m]\setminus\{i\}}=A_i$
for each $i$,
therefore $\bigcap_{i=1}^m A_i\ne\emptyset$.
This completes the proof in the case $m=n+2$.

The general case follows by induction in $m$. 
Suppose that $m>n+2$ and a collection $\{A_i\}_{i=1}^m$
satisfies the assumptions of the proposition. 
Then, since the case $m=n+2$ is already done,
every subcollection of cardinality $n+2$
has a nonempty intersection.
Therefore the collection $\{A'_i\}_{i=1}^{m-1}$
where $A'_i=A_i\cap A_m$ satisfies the assumptions as well.
Applying the induction hypothesis to $\{A'_i\}$ yields that the intersection
$\bigcap_{i=1}^{m-1}A'_i=\bigcap_{i=1}^m A_i$ is nonempty.
\end{proof}


\begin{proof}[Proof of Theorem \ref{t:main}]
In a uniquely geodesic space all nonempty convex sets are contractible.
This is ensured by the requirement that segments
depend continuously on their endpoints.
Intersections of convex sets are obviously convex
and hence contractible.
Therefore Theorem~\ref{t:main} follows from Proposition~\ref{p:topohelly}.
\end{proof}

\end{document}